\documentclass[10pt]{amsart}
\usepackage{graphicx}
\usepackage{latexsym}
\usepackage{amssymb}                
\usepackage{amsmath}
\usepackage{amsthm}
\usepackage{amscd}
\usepackage{enumerate}
\usepackage{tikz}
\usetikzlibrary{arrows}
\usepackage{float}
\usepackage{hyperref}

\theoremstyle{plain}
\newtheorem{thm}{Theorem}[section]
\newtheorem{lem}[thm]{Lemma}
\newtheorem{prop}[thm]{Proposition}
\newtheorem{cor}[thm]{Corollary}

\theoremstyle{definition}
\newtheorem{defi}[thm]{Definition}

\theoremstyle{remark}

\title{Centralizers in the Group of Interval Exchange Transformations}
\author{Daniel Bernazzani}

\begin{document}

\begin{abstract}
We study the group of interval exchange transformations. Let $T$ be an $m$-interval exchange transformation. By the rank of $T$ we mean the dimension of the $\mathbb{Q}$-vector space spanned by the lengths of the exchanged subintervals. We prove that if $T$ satisfies Keane's infinite distinct orbit condition and $\text{rank}(T)>1+\lfloor m/2 \rfloor$ then the only interval exchange transformations which commute with $T$ are its powers. 

In the case that $T$ is a minimal 3-interval exchange transformation, we prove a more precise result: $T$ has a trivial centralizer in the group of interval exchange transformations if and only if $T$ satisfies the infinite distinct orbit condition. 
\end{abstract}

\maketitle
\thispagestyle{empty}

\section{Introduction}

An interval exchange transformation (IET) is a bijective map $T:[0,1)\rightarrow [0,1)$ defined by partitioning the unit interval $[0,1)$ into finitely many subintervals and then rearranging these subintervals by translations. The formal definition appears below. The permutation group of the set $\lbrace 1,2,\dots,m \rbrace$ will be denoted by $S_m$. 

\begin{defi}\label{defi:iet}
Let $m\in \mathbb{N}$. Let $\pi \in S_m$ and let $\lambda=(\lambda_1,\lambda_2,\dots,\lambda_m)$ be a vector in the simplex $$\Delta_{m}=\bigg\lbrace (\lambda_1,\lambda_2,\dots,\lambda_{m})\in \mathbb{R}^m \hspace{1mm} : \hspace{1mm} \lambda_i >0, \hspace{1mm}\sum_i \lambda_i = 1 \bigg\rbrace.$$ Let $$\beta_0=0 \text{ and } \beta_j=\sum_{i=1}^j\lambda_j \text{ for } 1\leq j \leq m.$$ The set $\lbrace \beta_0,\beta_1,\dots,\beta_m \rbrace$ partitions $[0,1)$ into $m$ subintervals of the form $I_j=[\beta_{j-1},\beta_j)$. Define $T_{(\pi,\lambda)}:[0,1)\rightarrow [0,1)$ by $$T_{(\pi,\lambda)}(x)=x - \Bigg(\sum_{i<j}\lambda_i\Bigg) +\Bigg(\sum_{\pi(i)<\pi(j)}\lambda_i\Bigg), \text{ for } x\in I_j.$$ The map $T_{(\pi,\lambda)}$ rearranges the intervals $I_j$ by translations according to the permutation $\pi$. We will refer to a map constructed in this manner as an $m$-IET. For convenience, we sometimes drop the reference to $\pi$ and $\lambda$ and simply denote an IET by a single letter, typically $T$. 
\end{defi}  

The set of all IETs forms a group $\mathbb{G}$ under composition. Given $T\in \mathbb{G}$, let $C(T)$ denote the centralizer of $T$ in $\mathbb{G}$ and let $\langle T \rangle$ denote the cyclic subgroup generated by $T$. Novak proved that the quotient $C(T)/\langle T \rangle$ is typically finite \cite[Proposition~5.3]{Novak1}. However, there are examples of IETs which satisfy the hypotheses of Novak's theorem for which $C(T) \neq \langle T \rangle$ (some of these examples will be described in the next section of this paper). In fact, to the best of the author's knowledge, the only previously known examples where it could be proven that $C(T)=\langle T \rangle$ are due to del Junco \cite{delJunco}, who constructed IETs which do not commute with any Lebesgue measure preserving transformations of $[0,1)$ except for their powers.   

In this paper, we will show that for a large class of IETs, the groups $C(T)$ and $\langle T \rangle$ coincide. Recall that an IET $T$ is said to be minimal if for each $x\in [0,1)$, the orbit $\mathcal{O}_T(x)=\lbrace T^n(x) : n\in \mathbb{Z} \rbrace$ is dense in $[0,1)$. Recall also that an IET is said to be of rotation type if there exists $\alpha \in \mathbb{R}$ such that $T(x)=x+\alpha \hspace{2pt}(\text{mod }1)$ for all $x\in [0,1)$.

\begin{thm}\label{thm:centralizers general}
Let $T$ be a minimal IET which is not of rotation type. Suppose that the lengths of the exchanged subintervals are linearly independent over $\mathbb{Q}$. Then $C(T)=\langle T \rangle$.
\end{thm}

\noindent Theorem \ref{thm:centralizers general} is a simplified version of our main result, which will be stated in the next section of this paper (see Theorem \ref{thm:centralizer rank}).

We will denote $n$-fold compositions $T\circ \cdots \circ T$ by $T^n$. We will say that $T$ has an $n^{\text{th}}$ root in $\mathbb{G}$ if there exists $S\in \mathbb{G}$ such that $T=S^n$. If $T$ is minimal, then the existence of a $n^{\text{th}}$ root for some $n\geq 2$ implies that $C(T) \neq \langle T \rangle$. In a previous paper \cite{Bernazzani1}, the author proved that many IETs do not have any nontrivial roots in $\mathbb{G}$. In particular, it was proven that an IET which satisfies the hypotheses of Theorem \ref{thm:centralizers general} does not have an $n^{\text{th}}$ root for any $n\geq 2$. Though the non-existence of roots is weaker than the statement that $C(T)=\langle T \rangle$, our proofs of the results in this paper depend on the results in \cite{Bernazzani1}.

In the case that $T$ is a 3-IET, we will prove a more precise result. We recall the following definition. 
 
\begin{defi}\label{defi:idoc}
Let $T$ be an $m$-IET. Let $\beta_1,\beta_2,\dots,\beta_{m-1}$ be as in Definition 1.1. We say that $T$ satisfies the \textit{infinite distinct orbit condition} (IDOC) if each of the orbits $\mathcal O_T(\beta_1),\mathcal O_T(\beta_2),\dots,\mathcal O_T(\beta_{m-1})$ is infinite and $\mathcal O_T(\beta_i) \cap \mathcal O_T(\beta_j)=\emptyset$ for $i\neq j$.
\end{defi}

\noindent The IDOC was originally formulated by Keane, who showed that any IET which satisfies it and exchanges two or more intervals is minimal \cite{Keane1}. 

\begin{thm}\label{thm:centralizers idoc}
Let $T$ be a minimal 3-IET which is not of rotation type. Then $C(T)=\langle T \rangle$ if and only if $T$ satisfies the infinite distinct orbit condition. 
\end{thm}

Let $T$ be a 3-IET which is not of rotation type, and let $\lambda_1,\lambda_2,\lambda_3$ be the lengths of the exchanged subintervals. As explained in \cite[page~251]{Bernazzani1}, $T$ is minimal if and only if $\dfrac{\lambda_1 - \lambda_3}{1-\lambda_3} \notin \mathbb{Q}$. Moreover, assuming that $T$ is minimal, $T$ satisfies the IDOC if and only if $n(\lambda_1 -\lambda_3) \neq \lambda_1 + m(1-\lambda_3)$ for every pair of integers $m,n$. This enables one check whether or not $C(T)=\langle T \rangle$, provided one understands the rational dependencies among $\lambda_1,\lambda_2,\lambda_3$. 

\section{Centralizers in $\mathbb{G}$}

In this section we will define the rank of an IET and use this to state our main result, a generalization of Theorem \ref{thm:centralizers general}. We will also explain how our results are related to the work of Novak \cite{Novak1} and the previous work of the author \cite{Bernazzani1}.

The study of centralizers in $\mathbb{G}$ was initiated by Novak. Given $T\in \mathbb{G}$, let $d(T^n)$ denote the number of discontinuities of $T^n$. Novak showed that for a given $T\in \mathbb{G}$, there are only two possibilities for the growth rate of $d(T^n)$ \cite[Theorem~1.1]{Novak1}. The first possibility is that the growth rate is bounded: there exists a constant $M$ such that $d(T^n)\leq M$ for all $n\geq 1$. The second possibility is that $d(T^n)$ grows linearly: there exists a natural number $c$ such that $\displaystyle \lim_{n\rightarrow\infty} d(T^n)/n = c$. In the second case, Novak showed that $C(T)$ cannot be too large. 

\begin{thm}\label{thm:Novak}
(Novak \cite[Proposition~5.3]{Novak1}) Suppose that $T$ is a minimal IET and that $d(T^n)$ grows linearly in $n$. Then $\langle T \rangle$ has finite index in $C(T)$.
\end{thm}

\noindent An examination of the proof of Proposition 2.3 of \cite{Novak1} makes it clear that any IET $T$ satisfying the IDOC will exhibit linear discontinuity growth, provided that $T$ is not of rotation type. Novak makes a similar observation on p. 381 of his paper.

It is now convenient to introduce a definition. If $T$ is an $m$-IET, and $\beta_1,\dots,\beta_{m-1}$ are as in Definition \ref{defi:iet}, then it is clear that the discontinuities of $T$ must be among $ \beta_1,\dots,\beta_{m-1}$. However, it is possible that $T$ is continuous at some of these points. Whether or not $T$ is continuous at these points depends on the permutation $\pi$. More specifically, if $1 \leq i \leq m-1$, then $T_{(\pi,\lambda)}$ is discontinuous at $\beta_i$ if and only if $\pi(i+1) \neq \pi(i)+1$. This motivates the following definition.  

\begin{defi}\label{defi:separating}
We will say that $\pi \in S_m$ is \textit{separating} if $\pi(i+1)\neq \pi(i)+1$ for $1\leq i \leq m-1$.
\end{defi}

For example, the permutation $\tau =(4213)\in S_4$ is separating, while the permutation $\sigma = (4231) \in S_4$ is not. The following result, whose proof can be found in \cite[Proposition~2.3]{Bernazzani1}, shows that there is no loss of generality in only considering IETs defined by separating permutations. 

\begin{prop}\label{prop:separating existence}
Let $T$ be an IET with precisely $m-1$ discontinuities. There exists a separating permutation $\pi \in S_m$ and a vector $\lambda \in \Delta_m$, both of which are unique, such that $T=T_{(\pi,\lambda)}$.
\end{prop}

\noindent We will make use of the following special case of Theorem \ref{thm:Novak}.

\begin{prop}\label{prop:finite index}
Let $m\geq 3$. Suppose that $T$ is an $m$-IET defined by a separating permutation. If $T$ satisfies the IDOC, then $\langle T \rangle$ has finite index in $C(T)$. 
\end{prop}
\begin{proof}
Recall that any IET which satisfies the IDOC is minimal. Therefore, by Theorem \ref{thm:Novak} and the remarks immediately following it, we only need to show that $T$ is not of rotation type. Since the permutation defining $T$ is separating, $T$ has $m-1\geq 2$ discontinuities, so $T$ is not of rotation type.    
\end{proof}

It is natural to wonder under what circumstances we have equality $C(T)=\langle T \rangle$. The hypotheses of Theorem \ref{thm:Novak} are not enough to guarantee this. For example, let $S$ be a 3-IET with permutation $(321)$ which satisfies the IDOC. Let $T=S^n$ for some $n\geq 2$. Then $T$ is minimal and has linear discontinuity growth, but $C(T)\neq \langle T \rangle$, since $S\in C(T)$ and $S\not\in \langle T \rangle$.

As the preceding example shows, one obstruction to having $C(T)=\langle T \rangle$ is the existence of an $n^{\text{th}}$ root for some $n\geq 2$. In a previous paper, the author showed that most IETs do not have nontrivial roots. In order to state this result precisely, we recall the following definition.

\begin{defi}\label{defi:rank}
Let $T$ be an IET. Let $\gamma_1 < \gamma_2 < \cdots < \gamma_{m-1}$ be the points at which $T$ is discontinuous. Let $\gamma_0=0$ and $\gamma_m=1$. We will refer to the dimension of the $\mathbb{Q}$-vector space spanned by $\lbrace \gamma_i - \gamma_{i-1} : 1\leq i \leq m\rbrace$ as the \textit{rank} of $T$. This will be denoted by rank$(T)$.
\end{defi}

The term ``rank" was originally used in this setting by Boshernitzan \cite{Bosh_RankTwo}, who proved that minimal rank two IETs are uniquely ergodic and described an algorithm which tests a rank two IET for minimality and aperiodicity.

Using the notation of Definition \ref{defi:iet}, we observe that if $T_{(\pi,\lambda)}$ is an $m$-IET and $\pi$ is separating, then the discontinuities of $T$ are precisely the points $\beta_1,\beta_2,\dots,\beta_{m-1}$, so $\text{rank}(T)$ is equal to the dimension of the $\mathbb{Q}$-vector space spanned by $\lambda_1,\lambda_2,\dots,\lambda_m$. If $\pi$ is not separating, then $\text{rank}(T_{(\pi,\lambda)})$ may be smaller than the dimension of the $\mathbb{Q}$-vector space spanned by $\lambda_1,\lambda_2,\dots,\lambda_m$. 

Notice that if $T$ is an $m$-IET defined by a separating permutation, then $\text{rank}(T)$ can assume any value from $1$ to $m$. The following result was established by the author \cite[Theorem~2.4]{Bernazzani1}.

\begin{thm}\label{thm:roots rank}
Let $T$ be a minimal $m$-IET defined by a separating permutation. Suppose that rank$(T) > 1 +\lfloor m/2 \rfloor$. Then $T$ does not have an $n^{\text{th}}$ root in $\mathbb{G}$ for any $n\geq 2$.
\end{thm}

A second obstruction to having $C(T)=\langle T \rangle$ is the presence of torsion in $C(T)$. For convenience, we will give an example defined on $[0,2)$ rather than $[0,1)$. Let $S:[0,1)\rightarrow [0,1)$ be a 3-IET with permutation $(321)$ which satisfies the IDOC. Define $U:[0,2)\rightarrow [0,2)$ by $$U(x)=\begin{cases} S(x) & x\in [0,1) \\ x & x\in [1,2) \end{cases}$$ Let $P:[0,2)\rightarrow [0,2)$ be the 2-IET of order two which interchanges $[0,1)$ and $[1,2)$. Since the $U$ fixes $[1,2)$ and $PUP$ fixes $[0,1)$, $U$ and $PUP$ commute with one another. Consider the map $T=PUPUP=UPU$. It is not difficult to see that $T$ is minimal and that $d(T^n)\approx 2n$, so $T$ exhibits linear discontinuity growth. However, $P$ commutes with $T$ and $P\not\in \langle T \rangle$. 

One of the goals of this paper will be to establish that $C(T)$ is torsion-free under the following conditions. 

\begin{thm}\label{thm:torsion rank}
Let $T$ be an $m$-IET defined by a separating permutation. Suppose that $T$ satisfies the infinite distinct orbit condition and that rank$(T) > 1 +\lfloor m/2 \rfloor$. Then the only element of $C(T)$ which has finite order is the identity.
\end{thm}

For $n\in \mathbb{N}$, let $r_n$ denote the periodic IET defined by $r_n(x)= x + 1/n \hspace{2pt}(\text{mod }1)$. Suppose that $T$ is minimal and that $S\in C(T)$ has finite order $n$. Then $S$ is conjugate to $r_n$, and consequently $T$ is conjugate to some element of $C(r_n)$. Novak gives a precise description of $C(r_n)$ \cite[Proposition~5.5]{Novak1}, and it is clear from Novak's description that all of the elements of $C(r_n)$ fail to satisfy the estimate on the rank given in Theorem \ref{thm:torsion rank}. However, this argument does not prove Theorem \ref{thm:torsion rank} since conjugation need not preserve rank (see the \hyperref[sec:App]{Appendix} for an example). Instead of trying to circumvent this difficulty, we will prove Theorem \ref{thm:torsion rank} using a different technique. Incidentally, our proof can be used to show that there actually is a conjugating map from $S$ to $r_n$ which does not alter $\text{rank}(T)$.

The proof of Theorem \ref{thm:torsion rank} will be given in the fifth section of this paper. We will now show how to combine Proposition \ref{prop:finite index} and Theorems \ref{thm:roots rank} and \ref{thm:torsion rank} to establish our main result.

\begin{lem}\label{lem:trivial centralizer}
Let $G$ be a group and let $x\in G$. Suppose that  \begin{enumerate}
\item
$\langle x \rangle$ has finite index in $C(x)$;
\item
$x$ does not have an $n^{\text{th}}$ root in $G$ for any $n\geq 2$;
\item
$C(x)$ is torsion-free.
\end{enumerate}
Then $C(x)=\langle x \rangle$.
\end{lem}
\begin{proof}
Let $y\in C(x)$. By (1), there exists $n\in \mathbb{N}$ such that $y^n \in \langle x \rangle$, say $y^n = x^m$. We will prove that $y \in \langle x \rangle$ by induction on the number of primes dividing $n$. If there are no such primes, then $n=1$ and there is nothing to prove. Suppose that $n\geq 2$. Let $p$ be a prime dividing $n$, say $n=up$. Write $m=qp+r$, where $0\leq r < p$. We claim that $r=0$. If not, then $r$ and $p$ are relatively prime, so there exist integers $k$ and $l$ such that $kr-lp=1$. Then $$y^{upk}=y^{nk}=x^{mk}=x^{kqp + kr}=x^{kqp+lp +1}.$$ Since $x$ commutes with $y$, it follows that $\big(y^{uk}x^{-l-kq}\big)^p=x$. This contradicts (2). So $r=0$ and $m=qp$. We now have $y^{up}=y^n=x^m=x^{qp}$. After rearranging, we find that $\big(y^ux^{-q}\big)^p=e$, where $e$ denotes the identity in $G$. By (3), it must be that $y^ux^{-q}=e$. Therefore, $y^u=x^q$. Since $u = n/p$, it now follows by induction that $y\in \langle x \rangle$.
\end{proof}

\begin{thm}\label{thm:centralizer rank}
Let $T$ be an $m$-IET defined by a separating permutation. Suppose that $T$ satisfies the infinite distinct orbit condition and that rank$(T) > 1 +\lfloor m/2 \rfloor$. Then $C(T)=\langle T \rangle$.
\end{thm}
\begin{proof}
Recall that any IET which satisfies the IDOC is minimal. Notice also that we must have $m\geq 3$, since otherwise the inequality rank$(T) > 1 +\lfloor m/2 \rfloor$ could not possibly be true. The hypotheses of Proposition \ref{prop:finite index} and Theorems \ref{thm:roots rank} and \ref{thm:torsion rank} are all satisfied. The claim now follows from Lemma \ref{lem:trivial centralizer}.  
\end{proof}

Recall that a permutation $\pi \in S_m$ is said to be \textit{irreducible} if $\pi(\lbrace 1,2,\dots,k \rbrace) \neq \lbrace 1,2,\dots,k \rbrace$ for any $k<m$. A result of Keane asserts that if $\pi\in S_m$ is irreducible and the coordinates of $\lambda\in \Delta_m$ are linearly independent over $\mathbb{Q}$, then the IET $T_{(\pi,\lambda)}$ satisfies the IDOC \cite{Keane1}. We will also have $\text{rank}(T_{(\pi,\lambda)})=m$ in this case, provided that $\pi$ is separating. Combining these observations with Theorem \ref{thm:centralizer rank} proves the following result, which justifies the statement that ``most" IETs have trivial centralizers in $\mathbb{G}$. 

\begin{cor}
Let $m\geq 3$. Let $\pi \in S_m$ be separating and irreducible. Let $$A=\big\lbrace \lambda \in \Delta_m : C(T_{(\pi,\lambda)})=\langle T_{(\pi,\lambda)} \rangle \big\rbrace.$$ Then $A$ is a residual subset of $\Delta_m$ of full Lebesgue measure.
\end{cor}

\noindent For completeness, we explain why Theorem \ref{thm:centralizers general} follows from Theorem \ref{thm:centralizer rank}.\vspace{5pt}

\noindent \textit{Proof of Theorem \ref{thm:centralizers general}}. Let $T$ be a minimal IET which is not of rotation type. Suppose that the lengths of the exchanged subintervals are linearly independent over $\mathbb{Q}$. Let $m-1$ be the number of discontinuities of $T$. The assumption that $T$ is not of rotation type implies that $T$ has at least two discontinuities. So $m\geq 3$. Choose $\pi \in S_m$ and $\lambda\in \Delta_m$ according to Proposition \ref{prop:separating existence}. The assumption that the lengths of the exchanged subintervals are linearly independent over $\mathbb{Q}$ implies that the dimension of the $\mathbb{Q}$-vector space spanned by $\lambda_1,\lambda_2,\dots,\lambda_m$ is $m$. Since $T$ is minimal, it is clear that $\pi$ is irreducible. Therefore $T$ satisfies the IDOC. Moreover, since $\pi$ is separating, $\text{rank}(T)=m$. Since $m\geq 3$, $\text{rank}(T)> 1 + \lfloor m/2 \rfloor$, so Theorem \ref{thm:centralizer rank} implies that $C(T)=\langle T \rangle$. \qed
\vspace{5pt}

We now turn our attention to Theorem \ref{thm:centralizers idoc}. Our proof will follow the same strategy as our proof of Theorem \ref{thm:centralizer rank}. In place of Theorem \ref{thm:roots rank}, we will use following result, which was established by the author \cite[Theorem~1.5]{Bernazzani1}. 

\begin{thm}\label{thm:roots idoc}
Let $T$ be a minimal 3-IET which is not of rotation type. Then $T$ has an $n^{\text{th}}$ root in $\mathbb{G}$ for some $n\geq 2$ if and only if $T$ fails to satisfy the infinite distinct orbit condition.
\end{thm}

\noindent In place of Theorem \ref{thm:torsion rank}, we will use the following result. 

\begin{thm}\label{thm:torsion idoc}
Let $T$ be a 3-IET which satisfies the infinite distinct orbit condition and which is not of rotation type. Then $C(T)$ is torsion-free.
\end{thm}

The proof of Theorem \ref{thm:torsion idoc} will be given in the sixth section of this paper. We now proceed with the proof of Theorem \ref{thm:centralizers idoc}.\vspace{5pt}

\noindent \textit{Proof of Theorem \ref{thm:centralizers idoc}}. Let $T$ be a minimal 3-IET  which is not of rotation type. 

Suppose that $T$ satisfies the IDOC. The assumption that $T$ is not of rotation type implies that $T$ must be defined by the permutation $(321)$, which is separating. Thus the hypotheses of Proposition \ref{prop:finite index} and Theorems \ref{thm:roots idoc} and \ref{thm:torsion idoc} are all satisfied. An application of Lemma \ref{lem:trivial centralizer} shows that $C(T)=\langle T \rangle$.

Conversely, suppose that $T$ does not satisfy the IDOC. By Theorem \ref{thm:roots idoc}, there exists $S\in \mathbb{G}$ and $n\geq 2$ such that $T=S^n$. Since $T$ has infinite order in $\mathbb{G}$, it is clear that $S\not\in \langle T \rangle$. Therefore $C(T)\neq \langle T \rangle$.\qed
\vspace{5pt}

We remark that if $T$ is a minimal $3$-IET which is not of rotation type and which does not satisfy the IDOC, then Theorems 1.7 and 1.8 of \cite{Bernazzani1} can be used to show that $C(T)$ is actually uncountable and contains a subgroup isomorphic to the circle. We omit the details. 

The main results of this paper have been established, but the proofs of Theorems \ref{thm:torsion rank} and \ref{thm:torsion idoc} are still pending. The rest of this paper will be devoted to proving these two results. 

\section{The First Return Map}

In this section we briefly review some facts about first return maps which will be used in our proofs of Theorems \ref{thm:torsion rank} and \ref{thm:torsion idoc}. Given an IET $T$, let $D(T)$ denote the set of points at which $T$ is discontinuous. 
 
Let $\mathcal{A}$ denote the set algebra consisting of finite unions of half-open intervals $[a,b)$ contained in $[0,1)$. If $J\in \mathcal{A}$, let $E(J)$ denote the set of endpoints of the connected components of $J$. For example, if $J=[0,\frac{1}{3})\cup [\frac{4}{5},\frac{5}{6})$, then $E(J)=\lbrace 0,\frac{1}{3},\frac{4}{5},\frac{5}{6} \rbrace$. 

It is well-known that if $T$ is an IET and $J\in \mathcal{A}$ then the first return map to $J$ is essentially an IET. The following lemma is a precise formulation of this fact.  

\begin{lem}\label{lem:first return}
Let $T$ be an IET. Suppose that $J\in \mathcal{A}$ and that no point in $D(T)$ belong to the interior of $J$. Let $P\subseteq J$ consist of those points $x$ in the interior of $J$ for which there exists an $n\geq 1$ such that $T^j(x)\notin J \cup D(T) \cup E(J)$ for $0<j<n$ and $T^n(x)\in D(T) \cup E(J).$ The set $P$ is finite. Therefore $P$ partitions $J$ into finitely many subintervals
$J_1,J_2,\dots,J_k$. There exist positive integers $m_1,m_2,\dots,m_k$ such that 
\begin{enumerate}[(1)]
\item
for each $i$, the restriction of $T^j$ to $J_i$ is a translation for $1\leq j \leq m_i$;
\item
for each $i$, $T^j(J_i)\cap J = \emptyset$ for $0<j<m_i$;
\item
for each $i$, $T^{m_i}$ translates $J_i$ onto a subinterval of $J$;
\item
the intervals $T^{m_i}(J_i)$, $1\leq i \leq k$, are pairwise disjoint.
\end{enumerate} 
\end{lem}

\noindent For a proof of this result, see e.g. \cite[Lemma~4.2]{Viana}.

\begin{defi}
Let $T$ be an IET. Suppose that $J\in \mathcal{A}$ satisfies the hypothesis of Lemma \ref{lem:first return}. For each $x\in J$, let $n_J(x)= \inf \lbrace n \geq 1 :T^n(x)\in J \rbrace$. According to the lemma, $n_J(x)=m_i$ for $x\in J_i$. The map $T_J:J \rightarrow J$ defined by $T_J(x)=T^{n_J(x)}(x)$ is the \textit{first return map} to $J$. The integers $m_1,m_2,\dots,m_k$ are  the \textit{return times}. 
\end{defi}

For convenience, we will sometimes refer to first return maps as IETs, even though this is technically inconsistent with Definition \ref{defi:iet}, since the domain is not necessarily $[0,1)$.

\section{Fundamental Discontinuities}

In this section we use a result of Novak to derive some facts which will be used in our proofs of Theorems \ref{thm:torsion rank} and \ref{thm:torsion idoc}. 

Let $T$ be a minimal IET. Let $p\in D(T)$. Since $\mathcal{O}_T(p)$ is infinite, there  exists a smallest positive integer $N_p$ such that $T^{n}(p) \not\in D(T) \cup \lbrace 0 \rbrace$ for $n\geq N_p$. If $T^{N_p}$ is continuous at $p$, then all of the iterates $T^n$ for $n\geq N_p$ are continuous at $p$. On the other hand, if $T^{N_p}$ is discontinuous at $p$, then all of the iterates $T^n$ for $n\geq N_p$ are discontinuous at $p$. This dichotomy played an important role in Novak's study of the discontinuity growth rate of IETs \cite{Novak1}. 

\begin{defi}\label{defi:fundamental discontinuity}
Let $T$ be a minimal IET. We will say that $p\in D(T)$ is a \textit{fundamental discontinuity} if $T^n(p)\not\in D(T)$ for all negative $n$ and $T^n$ is discontinuous at $p$ for all sufficiently large positive $n$. We will denote the set of fundamental discontinuities by $D_f(T)$.
\end{defi}

We emphasize that Definition \ref{defi:fundamental discontinuity} implies that distinct fundamental discontinuities belong to different $T$-orbits. The following result is due to Novak. It is not stated as a separate lemma in \cite{Novak1}, but it is established in the proof of Proposition 5.3 in \cite{Novak1}.

\begin{lem}\label{lem:permute orbits}
(Novak) Suppose that $T$ is a minimal IET and that $S\in C(T)$. Then $S$ permutes the orbits of the fundamental discontinuities of $T$. More precisely, for each $x\in D_f(T)$, there exists a unique $y\in D_f(T)$ such that $S(\mathcal{O}_T(x))=\mathcal{O}_T(y)$.
\end{lem}

\begin{cor} \label{cor:cyclic group action}
Suppose that $T$ is a minimal IET and that $S\in C(T)$ has finite order $q$. Then $S$ induces a $\mathbb{Z}/q\mathbb{Z}$ action on the set $D_f(T)$. Specifically, the generator of $\mathbb{Z}/q\mathbb{Z}$ acts on $D_f(T)$ by $x\mapsto y$ where $y$ is the unique element of $D_f(T)$ with the property that $S(\mathcal{O}_T(x))=\mathcal{O}_T(y)$.
\end{cor}

It will be useful to know that each of the orbits of the $\mathbb{Z}/q\mathbb{Z}$ action described in Corollary \ref{cor:cyclic group action} has size $q$. We will prove this, but first we need the following lemma.

\begin{lem}\label{lem:fixed point centralizer}
Let $T$ be a minimal IET and let $S\in C(T)$. If $S$ has a fixed point, then $S$ is the identity map.
\end{lem}
\begin{proof}
Let $F$ be the set of points in $[0,1)$ which are fixed by $S$. The definition of an IET implies that $F$ is a finite union of half-open intervals. Since $T$ commutes with $S$, the set $F$ must be $T$-invariant. Since $T$ is minimal, $F=[0,1)$.
\end{proof}

\begin{cor}\label{cor:orbit size q}
Let $T$ be a minimal IET and let $S\in C(T)$. If $S$ has finite order $q$, then each of the orbits of the $\mathbb{Z}/q\mathbb{Z}$ action described in Corollary \ref{cor:cyclic group action} has size $q$.
\end{cor}
\begin{proof}
Let $x\in D_f(T)$ and suppose that the orbit of $x$ under the $\mathbb{Z}/q\mathbb{Z}$ action has size $u$, where $1\leq u<q$. Then $S^u(\mathcal{O}_T(x))=\mathcal{O}_T(x)$, so $S^u(x)=T^k(x)$ for some $k$. The map $S^uT^{-k}$ fixes $x$, so Lemma \ref{lem:fixed point centralizer} implies that $S^uT^{-k}=I$, where $I$ denotes the identity map. So $S^u=T^k$. Since $T$ has infinite order and $S^u$ has finite order, it must be that $k=0$. But then $S^u=I$, contrary to the fact that $S$ has order $q$. 
\end{proof}

The $\mathbb{Z}/q\mathbb{Z}$ action described above will play an important role in our proofs of Theorems \ref{thm:torsion rank} and \ref{thm:torsion idoc}. We will also use the following fact.

\begin{lem}\label{lem:idoc fundamental discontinuities}
Suppose that $T$ satisfies the IDOC. Then every discontinuity of $T$, with the possible exception of $T^{-1}(0)$, is a fundamental discontinuity.
\end{lem}
\begin{proof}
Let $p\in D(T)$ and suppose that $p\neq T^{-1}(0)$. Since $T$ satisfies the IDOC, $T^n(p)\not\in D(T)$ for all $n<0$. Let $n\geq 1$. Since $T$ satisfies the IDOC, $T$ is continuous at all points in the forward orbit of $p$. Since $T(p)$ is contained in the interior of $[0,1)$, it follows that $T^{n-1}$ is a translation on some open interval containing $T(p)$. Combining this with the fact that $T$ is discontinuous at $p$, it follows that $T^n$ is discontinuous at $p$. Since this is true for all $n\geq 1$, $p$ is a fundamental discontinuity of $T$.
\end{proof}
 
\section{Proof of Theorem \ref{thm:torsion rank}}

In this section we will prove Theorem \ref{thm:torsion rank}. It is convenient to introduce the following notation. Let $T$ be an IET. Given $x\in [0,1)$ and integers $n\leq m$, let $$\mathcal{O}_T(x,n,m) = \lbrace T^j(x) : n\leq j\leq m \rbrace.$$ Each set $\mathcal{O}_T(x,n,m)$ is just a finite part of the $T$-orbit of $x$. Notice that if $S$ is an IET which commutes with $T$, then $S(\mathcal{O}_T(x,n,m))= \mathcal{O}_T(S(x),n,m)$.

\begin{thm}\label{thm:rank estimate}
Suppose that $T$ is a minimal IET. Let $d_f$ denote the number of fundamental discontinuities. Let $d_{nf}$ denote the number of discontinuities which are not fundamental. If $S\in C(T)$ has finite order $q$, then 
$$\text{rank}(T)\leq \dfrac{d_f}{q} + d_{nf} + 1.$$
\end{thm}
\begin{proof}
For clarity, suppose first that every discontinuity of $T$ is fundamental. Consider the $\mathbb{Z}/q\mathbb{Z}$-action on $D_f(T)$ described in Corollary \ref{cor:cyclic group action}. According to Corollary \ref{cor:orbit size q}, the orbits of this action have size $q$. Let $x_1,x_2,\dots,x_{k}$ be representatives for the distinct $\mathbb{Z}/q\mathbb{Z}$-orbits in $D_f(T)=D(T)$. 

Using Lemma \ref{lem:fixed point centralizer}, it is not hard to see that every $S$-orbit has size $q$. Therefore there exists a subinterval $K=[a,b)\subseteq [0,1)$ such that the sets $S^i(K)$ for $0\leq i \leq q-1$ are pairwise disjoint and the restriction of $S$ to each of these sets is a translation. Notice that the set $J=\bigcup_{i=0}^{q-1}S^i(K)$ is $S$-invariant. We claim that if $K$ is chosen appropriately, then $J$ will not contain any of the discontinuities of $T$. 

In order to see this, let $N$ be a positive integer which will be specified later. By shrinking $K$ if necessary, we can assume that none of the points in the set $\bigcup_{i=1}^k \mathcal{O}_T(x_i,-N,N)$ belong to $J$. Since $J$ is $S$-invariant, it follows that none of the points in $\bigcup_{j=0}^{q-1} \bigcup_{i=1}^k S^j(\mathcal{O}_T(x_i,-N,N))$ belong to $J$. Each of the sets $S^j(\mathcal{O}_T(x_i,-N,N))=\mathcal{O}_T(S^j(x_i),-N,N)$ is part of the $T$-orbit of one of the discontinuities of $T$. Moreover, every discontinuity of $T$ belongs to one of the orbits $\mathcal{O}_T(S^j(x_i))$, where $1\leq i \leq k$ and $0\leq j \leq q-1$. It follows that if $N$ is chosen large enough, then all of the discontinuities of $T$ belong to $\bigcup_{j=0}^{q-1} \bigcup_{i=1}^k S^j(\mathcal{O}_T(x_i,-N,N))$ and therefore $D(T)\cap J = \emptyset$, as claimed. 

We can also ensure that the endpoints of the components of $J$ belong to the orbits of some of the points in $D(T)$. Indeed, let $M_1$ be the largest negative integer such that $T^{M_1}(x_1)\in J$. Let $M_2$ be the smallest positive integer such that $T^{M_2}(x_1) \in J$. By construction, $M_1 < -N$ and $M_2 > N$. By shrinking $K$, we can assume that the endpoints of the $q$ connected components of $J$ are precisely the points $S^i(T^{M_j}(x_1))$ for $j=1,2$ and $i=0,1,\dots,q-1$. Notice that none of the points in the set $\bigcup_{j=0}^{q-1} S^j(\mathcal{O}_T(x_1,M_1+1,M_2-1))$ belong to $J$.

We now consider the first return map $T_J$. We claim that the set $P$ described in Lemma \ref{lem:first return} consists of exactly $d_f$ points. By definition, the points belonging to $P$ are precisely those points in the interior of $J$ whose forward $T$-orbit encounters one of discontinuities of $T$ or one of the endpoints of $J$ before returning to $J$. By construction, a point has this property if and only if its forward $T$-orbit encounters one of the sets $S^j(\mathcal{O}_T(x_1,M_1+1,M_2-1))$, where $0\leq j \leq q-1$, or one of the sets $S^j(\mathcal{O}_T(x_i,-N,N))$, where $2\leq i \leq k$ and $0\leq j \leq q-1$, before returning to $J$. Clearly, the backward $T$-orbit of each discontinuity contains precisely one point with this property, so $\vert P \vert = d_f$ as claimed.

The discussion in the preceding paragraph makes it clear that the set $P$ is $S$-invariant. Indeed, since $S$ commutes with $T$, the forward $T$-orbit of a point $p\in J$ encounters one of the sets $S^j(\mathcal{O}_T(x_i,-N,N))$ or returns to $J$ at precisely the same time that the forward $T$-orbit of $S(p)$ encounters $S^{j+1}(\mathcal{O}_T(x_i,-N,N))$ or returns to $J$, respectively.      

Since $\vert P \vert = d_f$ and $J$ has $q$ connected components, $P$ partitions $J$ into $q+d_f$ subintervals $J_1,J_2,\dots,J_{q+d_f}$. Let $l_1,l_2,\dots,l_{q+d_f}$ denote the lengths of these intervals, respectively. Let $m_1,m_2,\dots,m_{q+d_f}$ be the return times for the map $T_J$. The set $A=\bigcup_{i=1}^{q+d_f} \bigcup_{j=0}^{m_i-1} T^j(J_i)$ is $T$-invariant. Since $T$ is minimal, it must be that $A=[0,1)$. Moreover, each of the intervals $T^j(J_i)$, for $1\leq i\leq q+d_f$ and $0\leq j \leq m_i-1$, must be contained in one of the intervals on which $T$ is continuous. For if this were not the case, then some point in the interior of $J$ would have to contain a discontinuity in its forward $T$-orbit. This point would then have to belong to $P$, a contradiction. It follows that each of the intervals on which are exchanged by $T$ must be a disjoint union of some of the intervals $T^j(J_i)$, where $1\leq i\leq q+d_f$ and $0\leq j \leq m_i-1$. These intervals are themselves just translates of the intervals $J_i$, for $1\leq i\leq q+d_f$. Therefore $\text{rank}(T)$ is at most equal to the number of distinct lengths appearing among $l_1,l_2,\dots,l_{q+d_f}$. 

Since the partition $P$ which determines the intervals $J_1,J_2,\dots,J_{q+d_f}$ is $S$-invariant, and $S$ acts on the $q$ connected components of $J$ by translation, it follows that there are at most $$\frac{q+d_f}{q}=\frac{d_f}{q}+1$$ 

\noindent different lengths appearing among $l_1,l_2,\dots,l_{q+d_f}$. It follows that $$\text{rank}(T)\leq \frac{d_f}{q}+1,$$ which is exactly what we wanted to show.

In the case when some of the discontinuities are not fundamental, we can proceed as we did above. The only essential difference is that for each non-fundamental discontinuity, there might be an additional point in the set $P$. This point divides one of the intervals $J_t$ into two intervals, and increases the rank of $T$ by at most one. Our previous estimate on the rank is therefore increased by at most $d_{nf}$, the number of non-fundamental discontinuities. Thus $$\text{rank}(T)\leq \frac{d_f}{q} + 1 + d_{nf}.$$ This completes the proof. 
\end{proof}

\noindent \textit{Proof of Theorem \ref{thm:torsion rank}}. Let $T$ be an $m$-IET defined by a separating permutation. Suppose that $T$ satisfies the IDOC. Assume that $S\in C(T)$ has finite order $q\geq 2$. We have to prove that $\text{rank}(T) \leq 1 + \lfloor m/2 \rfloor$.

Since $T$ is defined by a separating permutation, $D(T)=m-1$. Since $T$ satisfies the IDOC, Lemma \ref{lem:idoc fundamental discontinuities} tells us that there are two possibilities: either $d_{f}=m-1$ and $d_{nf}=0$ or else $d_{f}=m-2$ and $d_{nf}=1$. In the first case, Theorem \ref{thm:rank estimate} tells us that $$\text{rank}(T) \leq 1+ \frac{m-1}{q}  \leq 1+ \frac{m}{2}.$$ Since $\text{rank}(T)$ is an integer, it follows that $\text{rank}(T) \leq 1 + \lfloor m/2 \rfloor$. In the second case, Theorem \ref{thm:rank estimate} tells us that $$\text{rank}(T) \leq 2+ \frac{m-2}{q}  \leq 2+ \frac{m-2}{2}=1+\frac{m}{2}.$$ Once again, we conclude that $\text{rank}(T) \leq 1 + \lfloor m/2 \rfloor$.\qed

\section{Proof of Theorem \ref{thm:torsion idoc}}

Unlike our proof of Theorem \ref{thm:torsion rank}, which was combinatorial in nature, our proof of Theorem \ref{thm:torsion idoc} is based on the notion of topological weak-mixing. The following definition of topological weak-mixing will be sufficient for our purposes. 

\begin{defi}\label{defi:topological weak mixing}
Let $\lambda$ denote Lebesgue measure on $[0,1)$. If $T$ is an IET, then the map $f \mapsto f\circ T$ defines a unitary operator on $L^2([0,1),\lambda)$. We say that $T$ is \textit{topologically weak-mixing} if this operator has no non-constant piecewise continuous eigenfunctions. 
\end{defi}

To be clear, the term ``piecewise continuous" here means that $f:[0,1) \rightarrow \mathbb{C}$ has only finitely many discontinuities, that the left and right hand limits $\displaystyle\lim_{x \rightarrow a^-}f(x)$ and $\displaystyle\lim_{x \rightarrow a^+}f(x)$ exist for every $a\in (0,1)$, and that the limits $\displaystyle\lim_{x \rightarrow 0^+}f(x)$ and $\displaystyle\lim_{x \rightarrow 1^-}f(x)$ both exist. 

The next result follows from \cite[Theorem~4.1]{NogueiraRudolph}, but we include a proof for completeness.

\begin{lem}\label{lem:eigenfunction continuity}
Suppose that $T$ satisfies the IDOC (and exchanges two or more intervals). Then any piecewise continuous eigenfunction of $T$ must be continuous everywhere on $[0,1)$.
\end{lem}
\begin{proof}
Let $f$ be a piecewise continuous eigenfunction of $T$ and let $\alpha \in \mathbb{C}$ be the eigenvalue associated to $f$. 

Let $a\in (0,1)$ and let $\epsilon >0$. Let $y\in (0,1)$ be any point at which $f$ is continuous. Then there exists an $\eta>0$ such that $\vert f(p)-f(q) \vert < \epsilon$ whenever $p,q \in (y-\eta, y+\eta)$. Consider the sets $\lbrace T^k(a) : k \geq 0 \rbrace$ and $\lbrace T^k(a) : k < 0 \rbrace$. Since $T$ is minimal, both of these sets are dense in $[0,1)$. Since $T$ satisfies the IDOC, at least one of these sets does not contain any of the discontinuities of $T$. Suppose that $\lbrace T^k(a) : k \geq 0 \rbrace$ has this property. Then $T^k$ is continuous at $a$ for all $k>0$. Choose $k\geq 0$ such that $T^k(a)\in (y-\eta, y+\eta)$. Since $T^k$ is continuous at $a$, there exists $\delta>0$ such that $T^k(a-\delta, a+\delta) \subseteq (y-\eta, y+\eta)$. If $x\in (a-\delta, a+\delta)$, then both $T^k(x)$ and $T^k(a)$ belong to $(y-\eta, y+\eta)$. Using the fact that $\vert \alpha \vert =1$, we see that $$\vert f(x)-f(a) \vert = \vert \alpha^kf(x)- \alpha^kf(a) \vert = \vert f(T^k(x)) - f(T^k(a)) \vert < \epsilon. $$ This shows that $f$ is continuous at $a$. An analogous argument can be given when the set $\lbrace T^k(a) : k < 0 \rbrace$ contains no discontinuities of $T$. Continuity of $f$ at $0$ can be proven in a similar way.
\end{proof}

The following result is folklore. The author learned about it from Michael Boshernitzan.

\begin{prop}\label{prop:3-IET Topological weak mixing}
Let $T$ be a $3$-IET which satisfies the IDOC and which is not of rotation type. Then $T$ is topologically weak-mixing.
\end{prop}
\begin{proof}
Let $f:[0,1) \rightarrow \mathbb{C}$ be a piecewise continuous eigenfunction of $T$. Let $\alpha\in \mathbb{C}$ be the associated eigenvalue. By Lemma \ref{lem:eigenfunction continuity}, $f$ is continuous everywhere on $[0,1)$.

If $f$ vanishes at some point $x\in [0,1)$, then $f$ also vanishes along the $T$-orbit of $x$, which is dense in $[0,1)$. Since $f$ is continuous, this implies that $f$ is identically zero. So we may assume that $f$ vanishes nowhere. In particular, after multiplying $f$ by a suitable constant, we may assume that $f(0)=1$. Since $f$ is continuous at $0$, $\displaystyle\lim_{x\rightarrow 0^+} f(x) =1$.

Let $\beta_1$ and $\beta_2$ be the two points at which $T$ is discontinuous. Since, $T$ is not of rotation type, the permutation defining $T$ must be $(321)$. Therefore $T([\beta_2,1))$ lies to the the far left of the unit interval. It follows that $$\lim_{x\rightarrow \beta_2^+} f(x) = \lim_{x\rightarrow \beta_2^+} \frac{f(T(x))}{\alpha} = \lim_{x\rightarrow 0^+} \frac{f(x)}{\alpha} = \frac{1}{\alpha}.$$ We also know that the interval $T([\beta_1,\beta_2))$ lies immediately to the left of $T([0,\beta_1))$. Since $f$ is continuous at $T(0)$, it follows that $$\lim_{x\rightarrow \beta_2^-} f(x) = \lim_{x\rightarrow \beta_2^-} \frac{f(T(x))}{\alpha} = \lim_{x\rightarrow 0^+} \frac{f(T(x))}{\alpha} = \lim_{x\rightarrow 0^+} f(x) = 1.$$ Since $f$ is continuous at $\beta_2$, it must be that $$1 = \lim_{x\rightarrow \beta_2^-} f(x) = \lim_{x\rightarrow \beta_2^+} f(x) = \frac{1}{\alpha}.$$ So $\alpha=1$ and $f\circ T = f$. Since $T$ is minimal, $f$ must be a constant function.
\end{proof}

The next few results involve the concept of a tower over an IET, which we now explain. Let $T$ be an $m$-IET. Here we allow the domain of $T$ to be any finite interval $[a,b)$, possibly different from $[0,1)$. Let $I_1,I_2,\dots,I_m$ be the intervals which are exchanged by $T$ and suppose that $f:[a,b)\rightarrow \mathbb{N}$ is constant on each of these intervals, say $f(x)=n_j$ for $x\in I_j$. We can define a new map $T_f$ as follows. The domain will consist of those points of the form $(x,i)$ where $x\in [a,b)$ and $i$ is an integer satisfying $1\leq i\leq f(x)$. The map $T_f$ is defined by 
$$T_f(x,i)= \begin{cases} (x,i+1) & \text{if }i+1 \leq f(x) \\ (T(x),1) & \text{otherwise} \end{cases}$$ The domain may be visualized as tower over $[a,b)$. The map $T_f$ transports a point up to the next level of the tower, unless the point is already at the top, in which case it is transported back to the first level according to the original map $T$.

\begin{defi}\label{defi:tower}
Let $T$ and $f$ be as in the preceding paragraph. We can view the map $T_f$ as an IET by laying the levels of the tower end to end, and then rescaling so that the total length of the resulting interval is one. We will refer to an IET constructed in this manner as a tower of type $(n_1,n_2,\dots,n_m)$ over $T$. If all of the $n_i$ are equal, we will refer to $T_f$ as a tower of constant height $n_1=n_2=\cdots = n_m$. 
\end{defi}

\begin{lem}\label{lem:towers not topologically weak-mixing} Let $T$ be a tower over a minimal $2$-IET. Then $T$ is not topologically weak-mixing.
\end{lem}
\begin{proof}
By \cite[Theorem~1.8]{Bernazzani1}, $T$ is conjugate in the group $\mathbb{G}$ to either a minimal $2$-IET or a tower of constant height $d>1$ over a minimal $2$-IET. Let $g$ denote the conjugating IET in either case. 

Suppose first that $gTg^{-1}$ is a $2$-IET, say $x \mapsto x+ \alpha \hspace{2pt}(\text{mod }1)$. Let $f(x)= e^{2\pi ix}$. Then the function $f\circ g$ is a non-constant piecewise continuous eigenfuncion of $T$, with eigenvalue $e^{2\pi i \alpha}$.   

Now suppose that $gTg^{-1}$ is a a tower of constant height $d>1$ over a $2$-IET. Let $I_1,I_2,\dots,I_d$ be the levels of the tower and let $\chi_1,\chi_2,\dots,\chi_d$ be their characteristic functions, respectively. Let $\zeta_d$ be a primitive $d^{\text{th}}$ root of unity. Let $f(x)= \sum_{k=1}^d \zeta_d^{k} \chi_k(x)$. Then $f\circ g$ is a non-constant piecewise continuous eigenfunction of $T$, with eigenvalue $\zeta_d$. 
\end{proof}

\begin{prop}\label{prop:not weak mixing}
Let $T$ be a $3$-IET which satisfies the IDOC and which is not of rotation type. If $S\in C(T)$ has order two, then $T$ is not topologically weak-mixing.
\end{prop}
\begin{proof}
Let $\beta_1$ and $\beta_2$ be the discontinuities of $T$. According to Lemma \ref{lem:idoc fundamental discontinuities}, at least one of $\beta_1, \beta_2$ is a fundamental discontinuity. Since $S$ has order two, Corollary \ref{cor:orbit size q} implies that both $\beta_1$ and $\beta_2$ are fundamental discontinuities, and that $S$ permutes their orbits: $S(\mathcal{O}_T(\beta_1))=\mathcal{O}_T(\beta_2)$. As in the proof of Theorem \ref{thm:rank estimate}, there exists some interval $K\subseteq [0,1)$ such that 
\begin{enumerate}
\item
$K$ and $S(K)$ are disjoint;
\item
the set $J=K\cup S(K)$ satisfies the hypothesis of Lemma \ref{lem:first return};
\item
the set $P$ described in Lemma \ref{lem:first return} is $S$-invariant and contains precisely two points.
\end{enumerate}
It follows that the first return map $T_J$ exchanges four intervals, say $J_1,J_2 \subseteq K$ and $J_3=S(J_1),J_4=S(J_2)\subseteq S(K)$. Let $m_1,m_2,m_3,m_4$ be the return times. Since the set $J$ is $S$-invariant and $S$ commutes with $T$, the return time function $n_J(x)= \inf \lbrace n \geq 1 :T^n(x)\in J \rbrace$ must be $S$-invariant. Therefore, $m_1=m_3$ and $m_2=m_4$. This implies that the restriction of $S$ to $J$ commutes with $T_J$.

Let $\pi \in S_4$ be the permutation which describes how $T_J$ rearranges the intervals $J_1, J_2, J_3, J_4$. Notice that $S$ acts on $J$ according to the permutation $(13)(24)$. Since $S$ commutes with $T_J$, $\pi$ must commute with $(13)(24)$. So $\pi$ must be one of the following: the identity, $(13)$, $(24)$, $(13)(24)$, $(12)(34)$, $(14)(23)$, $(1234)$, or $(1432)$. We can rule out the first five possibilities, since these would result in $T_J$, and hence $T$, not being minimal.

Let $P$ be a tower of type $(m_1, m_2)$ over a rotation. Let $K_1$ and $K_2$ be the intervals which are exchanged by the underlying rotation. Notice that the lengths of $K_1$ and $K_2$ are twice those of $J_1$ and $J_2$, respectively. Let $g: [0,1) \rightarrow [0,1)$ be the two-to-one map which is defined as follows. For $0 \leq j \leq m_1 -1$, $g$ scales each of the intervals $T^j(J_1)$ and $T^j(J_3)$ by a factor of two, and then maps both onto $P^j(K_1)$. For $0 \leq j \leq m_2-1$, $g$ scales each of the intervals $T^j(J_2)$ and $T^j(J_4)$ by a factor of two, and then maps both onto $P^j(K_2)$. Using the fact that $\pi$ is one of $(1234)$, $(1432)$, $(14)(23)$, one can verify that $g\circ T = P \circ g$.

By Lemma \ref{lem:towers not topologically weak-mixing}, there exists a non-constant piecewise continuous eigenfunction, say $f$, of $P$. Let $\alpha \in \mathbb{C}$ be the associated eigenvalue. Then, for any $x\in [0,1)$, we have $f(g(T(x))) = f(P(g(x))) = \alpha f(g(x)).$ Since $g$ is locally an affine map, it is clear that $f\circ g$ is piecewise continuous and non-constant. This proves that $T$ is not topologically weak mixing.   
\end{proof}

\noindent \textit{Proof of Theorem \ref{thm:torsion idoc}.} Let $T$ be a 3-IET which satisfies the IDOC and which is not of rotation type. We have to prove that $C(T)$ is torsion-free.

Suppose that $S\in C(T)$ has finite order $q>1$. By Lemma \ref{lem:idoc fundamental discontinuities}, $T$ has at least one fundamental discontinuity. Since $T$ has at most two fundamental discontinuities, Corollary \ref{cor:orbit size q} implies $q \leq 2$, so it must be that $q=2$. Proposition \ref{prop:not weak mixing} now implies that $T$ is not topologically weak-mixing. This contradicts Proposition \ref{prop:3-IET Topological weak mixing}. \qed

\section{Appendix}\label{sec:App}
In the second section of this paper, we stated that conjugation in the group $\mathbb{G}$ need not preserve the rank of an IET. For the benefit of the reader, we provide here a simple example to illustrate this phenomenon. For convenience, our example will be defined on $[0,2)$ rather than $[0,1)$.

Given a real number $s$, let $R_s:[0,1) \rightarrow [0,1)$ be the rotation ($2$-IET) defined by $R_s(x) = x + s \hspace{2pt} (\text{mod } 1)$. Let $\alpha, \beta$ be two irrational numbers lying in the interval $(0,1)$. Assume furthermore that $\beta \not\in \mathbb{Q}(\alpha)$. Define a map $S: [0,2) \rightarrow [0,2)$ as follows: \[ S(x) = \begin{cases} R_\alpha(x) & \text{if } x\in [0,1) \\ R_\beta(x-1)+1 & \text{if } x\in [1,2) \end{cases} \] Let $P:[0,2) \rightarrow [0,2)$ be the $2$-IET which interchanges $[0,1)$ and $[1,2)$. Finally, let $T= PS$. The reader can verify that \begin{enumerate}
\item every point in $[0,1)$ returns to $[0,1)$ after exactly two iterations of $T$;
\item the first return map induced by $T$ on $[0,1)$ is $R_{\alpha + \beta}$.
\end{enumerate}
These two properties together imply that $T$ is conjugate to a tower of constant height two over $R_{\alpha + \beta}$. Such a map has rank two. However, since $\beta \not\in \mathbb{Q}(\alpha)$, $T$ has rank three.

\subsection*{Acknowledgments}
I would like to thank Michael Boshernitzan for encouraging me to investigate this topic, for reading many drafts of this paper, and for indicating how to prove Proposition \ref{prop:3-IET Topological weak mixing}. I would also like to thank the referee for reading the paper carefully and making several helpful suggestions.

\end{document}